\documentclass[12pt]{amsart}
\usepackage{amssymb}
\usepackage{amsmath}
\usepackage{amsthm}
\usepackage{amsbsy}
\usepackage{psfrag}
\usepackage{pstricks}
\usepackage{graphics}
\usepackage{graphicx}
\theoremstyle{plain}
\newtheorem{theorem}{Theorem}[section]
\newtheorem{lemma}[theorem]{Lemma}
\newtheorem{corollary}[theorem]{Corollary}

\theoremstyle{remark}
\newtheorem{remark}[theorem]{Remark}

\begin{document}
\allowdisplaybreaks[4]
\numberwithin{figure}{section}
\numberwithin{table}{section}
 \numberwithin{equation}{section}
%
\title[Adaptive Quadratic FEM for Obstacle Problem]
 {A Posteriori and A Priori  Error Estimates of Quadratic Finite Element Method for Elliptic Obstacle Problem}
\author{Thirupathi Gudi}\thanks{The author's work is supported in part by the grant from DST Fast Track Project and in part by the UGC Center for Advanced Study}
\address{Department of Mathematics, Indian Institute of Science, Bangalore - 560012}
\email{gudi@math.iisc.ernet.in}

\author{Kamana Porwal}\thanks{The second author's work is supported in part by
the UGC center for Advanced Study and in part by the Council for
Scientific and Industrial Research (CSIR)}
\address{Department of Mathematics, Indian Institute of Science, Bangalore - 560012}
\email{kamana@math.iisc.ernet.in}

\date{}
\begin{abstract}
A residual based {\em a posteriori} error estimator is derived for
a quadratic finite element method (fem) for the elliptic obstacle
problem. The error estimator involves various residuals consisting
the data of the problem, discrete solution and a Lagrange
multiplier related to the obstacle constraint. {\em A priori}
error estimates for the Lagrange multiplier have been derived and
further under an assumption that the contact set does not
degenerate to a curve in any part of the domain, optimal order
{\em a priori} error estimates have been derived whenever the data
and the solution are sufficiently regular, precisely, under the
sufficient conditions required for quadratic fem in the case of
linear elliptic problem. The numerical experiments of adaptive fem
for a model problem satisfying the above condition on contact set
show optimal order convergence. This demonstrates that the
quadratic fem for obstacle problem can exhibit optimal
performance.
\end{abstract}
\maketitle
\allowdisplaybreaks
\def\R{\mathbb{R}}
\def\cA{\mathcal{A}}
\def\cK{\mathcal{K}}
\def\cN{\mathcal{N}}
\def\p{\partial}
\def\O{\Omega}
\def\bbP{\mathbb{P}}
\def\cV{\mathcal{V}}
\def\cM{\mathcal{M}}
\def\cT{\mathcal{T}}
\def\cE{\mathcal{E}}
\def\bF{\mathbb{F}}
\def\bC{\mathbb{C}}
\def\bN{\mathbb{N}}
\def\ssT{{\scriptscriptstyle T}}
\def\HT{{H^2(\O,\cT_h)}}
\def\mean#1{\left\{\hskip -5pt\left\{#1\right\}\hskip -5pt\right\}}
\def\jump#1{\left[\hskip -3.5pt\left[#1\right]\hskip -3.5pt\right]}
\def\smean#1{\{\hskip -3pt\{#1\}\hskip -3pt\}}
\def\sjump#1{[\hskip -1.5pt[#1]\hskip -1.5pt]}
\def\jumptwo{\jump{\frac{\p^2 u_h}{\p n^2}}}

\section{Introduction}\label{sec:Intro}
Elliptic obstacle problem is a prototype model for the class of
elliptic variational inequalities of the first kind. The obstacle
problem is a nonlinear model describing the vertical displacement
of an object (with appropriate boundary conditions) constrained to
lie above an obstacle under a vertical force. The obstacle is a
given function with some smoothness. In general, the obstacle
problem exhibits a free boundary set (the boundary of the set
where the object touches the obstacle) where the regularity of the
solution being affected.  Therefore the numerical approximation of
an obstacle problem using uniform refinement will be inefficient.
Adaptive finite element methods (FEM), which compensate the
regularity, play an important role for these class of problems to
enhance the efficiency of the finite element method.

The application of finite element methods to obstacle problem
dates back to 1970's
\cite{Falk:1974:VI,BHR:1977:VI,Wang:2002:P2VI,Glowinski:2008:VI}.
The study of {\em a priori} error analysis for conforming linear
and quadratic finite element methods has been done in
\cite{Falk:1974:VI} and \cite{BHR:1977:VI}, respectively. In
\cite{Wang:2002:P2VI}, a refined error analysis for quadratic FEM
has been derived. The general convergence analysis and error
estimates for variational inequalities can be found in
\cite{Glowinski:2008:VI}. However, attention to the {\em a
posteriori} error analysis of FEMs for obstacle problem has begun
a decade and half ago
\cite{CN:2000:VI,Veeser:2001:VI,BC:2004:VI,Braess:2005:VI,NPZ:2010:VI,WW:2010:Apost}.
The first residual based {\em a posteriori} error estimator has
been derived in \cite{CN:2000:VI} for piecewise linear FEM  by
constructing a positivity preserving interpolation operator. The
{\em a posteriori} error analysis in \cite{Veeser:2001:VI} is
derived without using such a positivity preserving interpolation
operator. The error estimators in \cite{CN:2000:VI} and
\cite{Veeser:2001:VI} differ slightly from each other but it is
shown therein that both the estimators are reliable and efficient.
An averaging type error estimator is derived in \cite{BC:2004:VI}.
A simplified and abstract framework of error estimation for
conforming linear finite element methods can be found in
\cite{Braess:2005:VI,NPZ:2010:VI} when the obstacle is a global
affine function (Obstacle is $P_1(\Omega)$ function, where
$P_1(\Omega)$ is the space of linear polynomials restricted to
$\Omega$). The convergence of adaptive conforming linear finite
element method for obstacle problem is studied  in
\cite{BCH:2007:VI} for the first time. Recently discontinuous
Galerkin (DG) methods have been proposed and their {\em a priori}
error analysis has been derived in \cite{WHC:2010:DGVI}. More
recently, the {\em a posteriori} error analysis of linear DG
methods has been first studied in \cite{TG:2014:VIDG} and then
simplified in \cite{TG:2014:VIDG1}.

In this article, we derive a reliable {\em a posteriori} error
estimator for the quadratic FEM for elliptic obstacle problem. To
the best of our knowledge, this article is the first of such an
attempt. In the analysis, a discrete Lagrange multiplier is
introduced and used in a crucial way. {\em A priori} error
estimates derived in this article ensure the convergence of the
discrete Lagrange multiplier to the continuous one. Under an
assumption on the contact set, we derive optimal rate {\em a
priori} error estimates when the solution and the data are
sufficiently smooth. Numerical experiments for a model problem
with known solution illustrate the optimal rate of convergence
when adaptive algorithm is employed (since the solution of the
model problem is not regular enough, uniform refinement will not
yield optimal rate of convergence).

\par
\smallskip
Let $\Omega\subset \R^d (1\leq d \leq 3)$  be a bounded polyhedral
domain with boundary $\partial\Omega$. We assume that the obstacle
$\chi \in C(\bar\Omega)\cap H^1(\Omega)$ and satisfies
$\chi|_{\partial\Omega}\leq 0$. Then the closed and convex set
defined by
\begin{align*}
\cK=\{v\in H^1_0(\Omega): v\geq \chi \text{ a.e. in } \Omega\}
\end{align*}
is nonempty, since $\chi^+=\max\{\chi,0\}\in \cK$. The model
problem for the discussion below consists of finding $u\in \cK$
such that
\begin{align}\label{eq:MP}
a(u,v-u)\geq (f,v-u)\quad \forall v\in \cK,
\end{align}
where $a(u,v)=(\nabla u,\nabla v)$ and $f\in L^2(\Omega)$ is a
given function. Here and after, $(\cdot,\cdot)$ denotes the
$L^2(\Omega)$ inner-product while $\|\cdot\|$ denotes the
$L^2(\Omega)$ norm. The existence of a unique solution to
\eqref{eq:MP} follows from the result of Stampacchia
\cite{AH:2009:VI,Glowinski:2008:VI,KS:2000:VI}.

\par
\noindent For the error analysis, we need the Lagrange multiplier
 $\sigma\in H^{-1}(\Omega)$ defined by
\begin{align}\label{eq:sigmadef}
\langle \sigma, v\rangle =(f,v)-a(u,v)\quad \forall v\in
H^1_0(\Omega),
\end{align}
where $\langle \cdot,\cdot\rangle$ denotes the duality pair of
$H^{-1}(\Omega)$ and $H^1_0(\Omega)$. It follows from
\eqref{eq:sigmadef} and \eqref{eq:MP}, that
\begin{equation}\label{eq:sigma}
\langle \sigma, v-u\rangle\leq 0\quad  \forall v \in \cK.
\end{equation}

\par
The rest of the article is organized as follows. In Section
\ref{sec:DP}, we define the discrete problem and discuss
corresponding results. Section \ref{sec:Aposteriori} is devoted to
{\em a posteriori} error analysis. In Section \ref{sec:apriori},
we revisit the {\em a priori} error analysis of quadratic fem for
the obstacle problem, therein, we derive optimal order error
estimates for the solution and the Lagrange multiplier under some
conditions on the contact set. We present some numerical
experiments in Section \ref{sec:Numerics}. Finally, we conclude
the article in Section \ref{sec:Conclusions}.

\section{Discrete Problem}\label{sec:DP}
Below, we list the notation that will be used throughout the
article:
\begin{align*}
\cT_h &=\text{a regular simplicial triangulations of } \Omega\\
T&=\text{a triangle of } \cT_h,\qquad |T|=\text{ area of } T\\
 h_T &=\text{diameter of } T, \qquad h=\max\{h_T : T\in\cT_h\}\\
\cV_h^i&=\text{set of all vertices in } \cT_h \text{ that are in }
\O\\
\cV_T&=\text{set of three vertices of } T\\
\cE_h^i&=\text{set of all interior edges of } \cT_h\\
\cM_h^i&=\text{set of midpoints of interior edges in} \cT_h\\
\cM_T&=\text{set of midpoints of three edges of } T\\
h_e&=\text{length of an edge } e\in\cE_h.
\end{align*}
We mean by regular triangulation that there are no hanging nodes
in $\cT_h$ and the triangles in $\cT_h$ are shape-regular.
Further, we assume that each triangle $T$ in $\cT_h$ is closed.

In order to define the jump and mean of discontinuous functions
conveniently, define a broken Sobolev space
\begin{eqnarray*}
H^1(\O,\cT_h)=\{v\in L^2(\Omega) :\,v_{\ssT}= v|_{T}\in
H^1(T)\quad\forall\,~T\in{\mathcal T}_h\}.
\end{eqnarray*}
\par
For any $e\in\cE_h^i$, there are two triangles $T_+$ and $T_-$
such that $e=\partial T_+\cap\partial T_-$. Let $n_-$ be the unit
normal of $e$ pointing from $T_-$ to $T_+$, and $n_+=-n_-$. For
any $v\in H^1(\Omega,{\mathcal T}_h)$, we define the jump and mean
of $v$ on $e$ by
\begin{eqnarray*}
 \sjump{v} = v_-+v_+,\,\,\mbox{and}\,\,\smean{v} = \frac{1}{2}(v_-+v_+),\,\mbox{respectively,}
\end{eqnarray*}
 where $v_\pm=v\big|_{T_\pm}$.
Similarly define for $w\in H^1(\Omega,{\mathcal T}_h)^2$ the jump
and mean of $w$ on $e\in\cE_h^i$ by
\begin{eqnarray*}
 \sjump{w} = w_-\cdot n_-+w_+\cdot n_+,\,\,\mbox{and}\,\,\smean{w} = \frac{1}{2}(w_-+w_+),\,\mbox{respectively,}
\end{eqnarray*}
 where $w_\pm=w |_{T_\pm}$.

\par
For any edge $e\in \cE_h^b$,  there is a triangle $T\in\cT_h$ such
that $e=\partial T\cap \partial\Omega$. Let $n_e$ be the unit
normal of $e$ that points outside $T$. For any $v\in H^1(T)$, we
set on $e\in \cE_h^b$
\begin{eqnarray*}
 \sjump{v}= v,\,\,\mbox{and}\,\,\smean{v}=v,
\end{eqnarray*}
and for $w\in H^1(T)^2$,
\begin{eqnarray*}
\sjump{w}= w\cdot n_e,\,\,\mbox{and}\,\,\smean{w}=w.
\end{eqnarray*}

\subsection{Discrete Spaces}
The quadratic finite element space $V_h$ is defined by
$$V_h=\{v_h\in H^1_0(\Omega): v_h|_{T} \in \bbP_2(T) \quad \forall T\in\cT_h\}.$$
For the convenience of subsequent discussion, let $\{\psi_z,
z\in\cV_h^i\cup\cM_h^i\}$ be the canonical basis of $V_h$, i.e,
for $q\in\cV_h^i\cup\cM_h^i$
\begin{equation*}
\psi_z (q) := \left\{ \begin{array}{ll} 1 & \text{ if } z = q,\\\\
0 & \text{otherwise}.
\end{array}\right.
\end{equation*}
We need some more discrete spaces for the analysis to be followed.
Let
\begin{align*}
W_h=\text{Span}\{\psi_z\in V_h : z\in \cM_h^i\}.
\end{align*}
The subspace $W_h^c$ of $V_h$ defined by
\begin{align*}
W_h^c=\text{Span}\{\psi_z\in V_h : z\in \cV_h^i\},
\end{align*}
is the orthogonal complement of $W_h$ in $V_h$ with respect to the
inner product:
\begin{equation*}
\langle w_h,v_h\rangle_{V_h} =
\sum_{T\in\cT_h}\frac{|T|}{3}\left(\sum_{z\in\cM_T}w_h(z)v_h(z)+\sum_{z\in\cV_T}w_h(z)v_h(z)\right).
\end{equation*}
Then, we have $V_h=W_h\oplus W_h^c$. Let $V_{nc}$ be the
Crouziex-Raviart $P_1$-nonconforming space defined by
\begin{equation*}
V_{nc}=\{v_h\in L^2(\Omega) : v_h|_T\in P_1(T)\,\,\forall\, T\in
{\mathcal T}_h,\; \sjump{v}(z)=0\,\,\forall\,z\in \cM_h\}.
\end{equation*}
Let $\{\phi_z, z\in\cM_h^i\}$  be the canonical basis of $V_{nc}$,
i.e, for $q\in\cM_h^i$
\begin{equation*}
\phi_z (q) := \left\{ \begin{array}{ll} 1 & \text{ if } z = q,\\\\
0 & \text{otherwise}.
\end{array}\right.
\end{equation*}
Define an interpolation  $\Pi_h:V_{nc}\rightarrow W_h$ by
\begin{equation*}
\Pi_h v=\sum_{z\in\cM_h^i} v(z)\psi_z, \quad v\in V_{nc},
\end{equation*}
where $\psi_z\in W_h$ is a canonical basis function of $W_h$. It
holds that
\begin{equation*}
\Pi_h v(z)=v(z)\text{ for } z\in \cM_h \text{ and } v\in V_{nc}.
\end{equation*}
Note that $\Pi_h:V_{nc}\rightarrow W_h$ is bijective and hence its
inverse $\Pi^{-1}_h:W_h\rightarrow V_{nc}$ exists and it is given
by
\begin{equation*}
\Pi_h^{-1} v=\sum_{z\in\cM_h^i} v(z)\phi_z, \quad v\in W_h,\quad
\phi_z\in V_{nc},
\end{equation*}
where $\phi_z\in V_{nc}$ is a canonical basis function of
$V_{nc}$. Indeed $\Pi_h^{-1}$ extends to whole $V_h$  by defining
\begin{equation*}
\Pi_h^{-1} v=\sum_{z\in\cM_h^i} v(z)\phi_z, \quad v\in V_h.
\end{equation*}
For any $v\in V_h$, let $v=v_1+v_2$ where $v_1\in W_h$ and $v_2\in
W_h^c$. Then it is clear that  $\Pi_h^{-1} v=\Pi_h^{-1} v_1$.

\par
\noindent The following lemma determines the approximation
properties of $\Pi_h^{-1}$.

\begin{lemma}\label{lem:SigmaQuad} It holds that
\begin{equation*}
\|v_h-\Pi^{-1}_hv_h\|_{L^2(T)}\leq  C h_T\|\nabla
v_h\|_{L^2(T)}\quad \forall\,v_h\in V_h.
\end{equation*}
\end{lemma}
\begin{proof}
Using the Bramble-Hilbert Lemma \cite{BScott:2008:FEM} and an
inverse inequality, we find
\begin{equation*}
\|v_h-\Pi^{-1}_hv_h\|_{L^2(T)}\leq  C h_T^2|v_h|_{H^2(T)}\leq C
h_T\|\nabla v_h\|_{L^2(T)}.
\end{equation*}
This completes the proof.
\end{proof}

\subsection{Discrete Problem}
Define the discrete set
\begin{align}\label{eq:Kh}
\cK_h=\{v_h\in V_h: v_h(z)\geq \chi(z)\quad \forall z\in\cM_h\}.
\end{align}
The discrete problem consists of finding $u_h\in\cK_h$ such that
\begin{align}\label{eq:DP}
a(u_h,v_h-u_h)\geq (f,v_h-u_h)\quad \forall v_h\in \cK_h.
\end{align}
This method is introduced in \cite{BHR:1977:VI}. As in the case of
continuous problem \eqref{eq:MP}, the discrete problem
\eqref{eq:DP} has a unique solution.
\par
Note that for any $z\in \cV_h^i$ and the corresponding basis
function $\psi_z$, we have $\pm\psi_z \in \cK_h$. Then
\eqref{eq:DP} implies
\begin{align}\label{eq:DPProperty}
a(u_h,\psi_z)=(f,\psi_z) \quad \text{ for } z \in \cV_h^i.
\end{align}
Taking $v=u_h+\psi_z$ for any $z\in \cM_h^i$, we find from
\eqref{eq:DP} that
\begin{align}\label{eq:DPProperty1}
a(u_h,\psi_z)\geq (f,\psi_z) \quad \text{ for } z \in \cM_h^i.
\end{align}
Furthermore, it holds that
\begin{align}\label{eq:DPProperty2}
a(u_h,\psi_z)=(f,\psi_z) \quad \text{ for } z \in \{z \in\cM_h^i
:u_h(z)>\chi(z)\}.
\end{align}

\section{A Posteriori Error Estimates}\label{sec:Aposteriori}
\par
In the {\em a posteriori} error analysis below, we require a
discrete Lagrange multiplier $\sigma_h$ analogous to $\sigma$ in
\eqref{eq:sigmadef}.
Define $\sigma_h\in V_{nc}$ by
\begin{align}\label{eq:sigmahdef}
\langle\sigma_h,v_h\rangle_h=(f,\Pi_hv_h)-a(u_h,\Pi_hv_h)\quad
\forall v_h\in V_{nc},
\end{align}
where $\langle\cdot,\cdot\rangle_h$ is defined by
\begin{equation*}
\langle w_h,v_h\rangle_h =
\sum_{T\in\cT_h}\frac{|T|}{3}\sum_{z\in\cM_T}w_h(z)v_h(z).
\end{equation*}
Since $\langle \cdot,\cdot\rangle_h$ defines an inner-product on
$V_{nc}$, we have $\sigma_h$ well-defined. From \cite[Chapter
4]{Ciarlet:1978:FEM}, note that
\begin{equation}\label{eq:QuadExact}
\int_T v_h\,dx = \frac{|T|}{3}\sum_{z\in\cM_T}v_h(z)\quad \forall
v_h\in P_2(T).
\end{equation}

\begin{remark}
The choice of $\sigma_h$ in \eqref{eq:sigmahdef} is motivated by
two facts. First, the discrete set $\cK_h$ has constraints at the
midpoints of all the edges. This fact should be incorporated in
the definition of $\sigma_h$ which can be seen in the properties
of $\sigma_h$ in Lemma \ref{lem:propssigma} below. These
properties are very useful in our {\em a posteriori} error
analysis. Second, $\sigma_h$ should be a good approximation of
$\sigma$. The approximation properties of $\sigma_h$ in Section
\ref{sec:apriori} realizes this.
\end{remark}

\par
\noindent In the following lemma, we derive some useful properties
of $\sigma_h$:

\begin{lemma}\label{lem:propssigma} There hold
\begin{align}
\sigma_h(z) &\leq 0\quad \forall z\in\cM_h^i,
\label{eq:sigmah-at-z}\\
 \sigma_h(z) & = 0 \quad \text{ for } z \in
\{z \in\cM_h^i :u_h(z)>\chi(z)\}.\label{eq:sigmah-at-zz}
\end{align}
\end{lemma}
\begin{proof} Taking $v_h=\Pi_h^{-1}\psi_z$ for $z\in\cM_h^i$ in
\eqref{eq:sigmahdef} and using \eqref{eq:DPProperty1}, we find
that $\sigma_h(z)\leq 0$ for any $z\in\cM_h^i$. If
$u_h(z)>\chi(z)$ for any $z\in\cM_h^i$, then it holds from
\eqref{eq:DPProperty2} that $\sigma_h(z)= 0$. This completes the
proof.
\end{proof}

\par
\noindent For given $v\in L^2(\Omega)$, define the piecewise
constant (with respect to the triangulation) approximation $\bar
v$  by the following:
\begin{align*}
\bar v|_T=\frac{1}{|T|}\int_T v\,dx.
\end{align*}
It is well-known from \cite{BScott:2008:FEM,Ciarlet:1978:FEM} that
\begin{align*}
\|v_h-\bar v_h\|_{L^2(T)}\leq C h_T \|\nabla v\|_{L^2(T)} \quad
\forall\; v\in H^1(T),\quad\forall \;T\in\cT_h.
\end{align*}

\par
\noindent Define the following sets:
\begin{align*}
\bC_h=\{ T\in\cT_h: \text{For all } z\in\cM_T,\; u_h(z)=\chi(z)\},
\end{align*}
\begin{align*}
\bN_h=\{ T\in\cT_h: \text{For all } z\in\cM_T,\; u_h(z)>\chi(z)\},
\end{align*}
and
\begin{align*}
\bF_h=\{ T\in\cT_h: \exists\; z_1,\;z_2\in\cM_T \text{ such that }
u_h(z_1)=\chi(z_1) \text{ and } u_h(z_2)>\chi(z_2)\}.
\end{align*}
We call $\bC_h$, $\bN_h$ and $\bF_h$ as contact, non-contact and
free boundary set, respectively.

\par
\smallskip
\noindent Define the following estimators:
\begin{align*}
\eta_1&=\left(\sum_{T\in \cT_h} h_T^2 \|\Delta u_h+f-\sigma_h\|_{L^2(T)}^2\right)^{1/2},\\
 \eta_2&=\left(\sum_{e\in\cE_h^i}h_e\|\sjump{\nabla u_h}\|_{L^2(e)}^2\right)^{1/2},\\
\eta_3&=\left(\sum_{T\in \cT_h} h_T^2
\|\sigma_h-\bar\sigma_h\|_{L^2(T)}^2\right)^{1/2},\\
 \text{ and }\quad  \eta_4&=Osc(f,{\cT}_h),
\end{align*}
where the oscillations $Osc(f,{\mathcal D}_h)$ of $f$ over
${\mathcal D}_h\subseteq\cT_h$  is defined by
$$Osc(f,{\mathcal D}_h)=\left(\sum_{T\in {\mathcal D}_h} h_T^2\min_{c\in
P_0(T)}\|f-c\|_{L^2(T)}^2\right)^{1/2}.$$
The following lemma is a consequence of \eqref{eq:QuadExact} and
Lemma \ref{lem:propssigma}:
\begin{lemma}\label{lem:sigmaprops}
There hold
\begin{align*}
\bar\sigma_h &\leq 0 \text{ everywhere on } \Omega,\\
\bar\sigma_h &=0 \text{ on } \bN_h.
\end{align*}
\end{lemma}

Below, we derive a relation between the estimators.

\begin{lemma}\label{lem:sigmaosc}
Let $T\in \cT_h$. Then
\begin{equation*}
\eta_3 \leq C \left(\eta_1+\eta_4\right).
\end{equation*}
\end{lemma}
\begin{proof}
Since $u_h$ is piecewise quadratic, we find using triangle
inequality and the stability of the $L^2$-projection that
\begin{align*}
\|\bar\sigma_h-\sigma_h\|_{L^2(T)} &\leq \|f+\Delta
u_h-\sigma_h\|_{L^2(T)}+\|\bar\sigma_h-\bar f-\Delta
u_h\|_{L^2(T)}+\|f-\bar f\|_{L^2(T)}\\
&=\|f+\Delta u_h-\sigma_h\|_{L^2(T)}+ \|\bar\sigma_h-\bar
f-\overline{\Delta u_h}\,\|_{L^2(T)}+\|f-\bar f\|_{L^2(T)}\\
&=\|f+\Delta u_h-\sigma_h\|_{L^2(T)}+ \|\overline{\sigma_h-
f-\Delta u_h}\,\|_{L^2(T)}+\|f-\bar f\|_{L^2(T)}\\
&\leq 2 \|f+\Delta u_h-\sigma_h\|_{L^2(T)}+ \inf_{c\in
P_0(T)}\|f-c\|_{L^2(T)}.
\end{align*}
This completes the proof.
\end{proof}

\par
\noindent Define the residual $G_h: H^1_0(\Omega)\rightarrow \R$
by
\begin{align}\label{eq:Gh}
G_h(v)=a(u-u_h,v)+\langle \sigma-\sigma_h, v\rangle \quad \forall
v\in H^1_0(\Omega).
\end{align}
The residual $G_h$ helps to derive the error estimates as in the
case of linear elliptic problems. It is easy to prove the
following lemma which connects the norm of the error to the norm
of the residual $G_h$.

\begin{lemma}\label{lem:ErrorRelation}
There hold
\begin{align*}
\|\nabla(u-u_h)\|^2 &\leq \|G_h\|_{-1}^2-2\langle \sigma-\sigma_h,
u-u_h\rangle,\\
 \|\sigma-\sigma_h\|_{-1}&\leq \|G_h\|_{-1}+\|\nabla(u-u_h)\|,\\
\|\nabla(u-u_h)\|^2+\|\sigma-\sigma_h\|_{-1}^2&\leq
4\|G_h\|_{-1}^2-4\langle \sigma-\sigma_h, u-u_h\rangle.
\end{align*}
\end{lemma}

In the following lemma, the norm of the residual $G_h$ has been
estimated using error estimators:

\begin{lemma}\label{lem:GhBound}
It holds that
\begin{align*}
\|G_h\|_{-1}\leq C \left(\eta_1^2+\eta_2^2+\eta_3^2\right)^{1/2}.
\end{align*}
\end{lemma}

\begin{proof}
Let $v\in H^1_0(\Omega)$ and $v_h\in V_h$ be the Clement
interpolation of $v$. Then
\begin{align}
\langle G_h,v\rangle&=\langle G_h,v-v_h\rangle+\langle
G_h,v_h\rangle.
\end{align}
Note that any $v_h\in V_h$ can be written as
$$v_h=v_1+v_2, \quad v_1\in W_h \text{ and } v_2 \in W_h^c.$$
Firstly using \eqref{eq:Gh}, \eqref{eq:sigmadef},
\eqref{eq:sigmahdef} and \eqref{eq:DPProperty}, we find
\begin{align*}
\langle G_h, v_h\rangle&=a(u-u_h, v_h)+\langle \sigma-\sigma_h, v_h \rangle\\
&=-a(u_h, v_h)-(\sigma_h, v_h)+a(u, v_h)+\langle \sigma, v_h \rangle\\
&=(f, v_h)-a(u_h, v_h)-(\sigma_h, v_h)\\
&=(f, v_1)-a(u_h, v_1)-(\sigma_h, v_h)\\
&=\langle \sigma_h, \Pi^{-1}_hv_1 \rangle_h-(\sigma_h, v_h)
=(\sigma_h, \Pi^{-1}_hv_1-v_h)\\
&=(\sigma_h, \Pi^{-1}_hv_h-v_h)=(\sigma_h-\bar\sigma_h,
\Pi^{-1}_hv_h-v_h),
\end{align*}
since $(\bar\sigma_h, v_h-\Pi^{-1}_hv_h)=0$. Using Lemma
\ref{lem:SigmaQuad}
\begin{align*}
\langle G_h,v_h\rangle \leq C\eta_3 \|\nabla v_h\|\leq C\eta_3
\|\nabla v\|.
\end{align*}
Secondly using \eqref{eq:Gh},  we find
\begin{align*}
\langle G_h,v-v_h\rangle&=a(u-u_h,v-v_h)+\langle \sigma-\sigma_h,v-v_h \rangle\\
&=(f,v-v_h)-a(u_h,v-v_h)-(\sigma_h,v-v_h)\\
&=\sum_{T\in\cT_h} \int_{T}(f+\Delta u_h-\sigma_h)(v-v_h)\,dx-\sum_{T\in\cT_h} \int_{\partial T}\frac{\partial u_h|_T}{\partial n_T}(v-v_h)\,ds\\
&=\sum_{T\in\cT_h} \int_{T}(f+\Delta u_h-\sigma_h)(v-v_h)\,dx-\sum_{e\in\cE_h^i}\int_e\sjump{\nabla u_h}(v-v_h)\,ds\\
&\leq C \left(\eta_1^2+\eta_2^2\right)^{1/2}\|\nabla v\|.
\end{align*}
This completes the proof.
\end{proof}

It remains to find a lower bound for $\langle \sigma-\sigma_h,
u-u_h\rangle$. To this end,  let  $v^+=\max\{v,0\}$ and
$v^-=\max\{-v,0\}$ for any $v\in H^1(\Omega)$. Then $v=v^+ - v^-$.
For the rest of the article, the $P_2$-Lagrange nodal
interpolation of $\chi$ in $V_h$ is denoted by $\chi_h$.
\begin{lemma}\label{lem:sigmaLower}
Let $\chi_h\in V_h$ be the Lagrange nodal interpolation of $\chi$.
Then, it holds that
\begin{align*}
\langle \sigma-\sigma_h, u-u_h\rangle &\geq -\frac{1}{8}
\left(\|\sigma-\sigma_h\|_{-1}^2+\|\nabla(u-u_h)\|^2\right)-C\left(\|\nabla(\chi-u_h)^+\|^2+\eta_3^2\right)\\
&\quad+\sum_{T\in\bF_h}\int_T \bar
\sigma_h(u_h-\chi_h)\,dx+\sum_{T\in\bC_h\cup\,\bF_h}\int_T\bar
\sigma_h(\chi_h-\min\{u_h,\chi\})\,dx.
\end{align*}
\end{lemma}
\begin{proof}
Let $u_h^*=\max\{u_h,\chi\}$. Then $u_h^*\in\cK$ and
$u_h^*-u_h=(\chi-u_h)^+$. Using \eqref{eq:sigma} and $ab\leq
2a^2+b^2/8$, we find
\begin{align*}
\langle \sigma, u-u_h\rangle &= \langle \sigma,
u-u_h^*\rangle+\langle \sigma, u_h^*-u_h\rangle \geq \langle
\sigma, u_h^*-u_h\rangle\\
& = \langle \sigma-\sigma_h, u_h^*-u_h\rangle + \langle \sigma_h,
u_h^*-u_h\rangle\\
&\geq -\frac{1}{8}
\|\sigma-\sigma_h\|_{-1}^2-2\|\nabla(u_h^*-u_h)\|^2 + \langle
\sigma_h, u_h^*-u_h\rangle.
\end{align*}
\par
\noindent Therefore
\begin{align*}
\langle \sigma-\sigma_h, u-u_h\rangle &\geq -\frac{1}{8}
\|\sigma-\sigma_h\|_{-1}^2-2\|\nabla(u_h^*-u_h)\|^2
+ \langle \sigma_h, (u_h^*-u_h)-(u-u_h)\rangle\\
&\geq -\frac{1}{8}
\|\sigma-\sigma_h\|_{-1}^2-2\|\nabla(u_h^*-u_h)\|^2
+ \langle \bar  \sigma_h, (u_h^*-u_h)-(u-u_h)\rangle\\
&\quad+\langle \sigma_h-\bar  \sigma_h,
(u_h^*-u_h)-(u-u_h)\rangle.
\end{align*}
Notice that
\begin{align*}
|\langle \sigma_h-\bar  \sigma_h, (u_h^*-u_h)-(u-u_h)\rangle |&=
|\langle \sigma_h-\bar  \sigma_h,
(u_h^*-u_h)-\overline{(u_h^*-u_h)}+(u-u_h)-\overline{(u-u_h)}\rangle|
\\& \leq C \eta_3
\left(\|\nabla(u_h^*-u_h)\|+\|\nabla(u-u_h)\|\right)\\
&\leq C
\left(\eta_3^2+\|\nabla(u_h^*-u_h)\|^2\right)+\frac{1}{8}\|\nabla(u-u_h)\|^2.
\end{align*}
Using the fact $\chi-u\leq 0$ a.e. in $\Omega$ and $\bar
\sigma_h\leq 0$ on $\bar\Omega$, we find
\begin{align*}
\langle \bar  \sigma_h, (u_h^*-u_h)-(u-u_h)\rangle \geq \langle
\bar  \sigma_h, (u_h^*-u_h)-(\chi-u_h)\rangle.
\end{align*}
Note that $(u_h^*-u_h)-(\chi-u_h)=(\chi-u_h)^-$ and
\begin{align*}
\langle \bar  \sigma_h, (u_h^*-u_h)-(u-u_h)\rangle \geq \langle
\bar  \sigma_h, (\chi-u_h)^-\rangle=\int_\Omega \bar
\sigma_h(\chi-u_h)^-\,dx.
\end{align*}
Now using Lemma \ref{lem:sigmaprops},
\begin{align*}
\int_\Omega \bar \sigma_h(\chi-u_h)^-\,dx=
\sum_{T\in\bC_h\cup\,\bF_h}\int_T\bar \sigma_h(\chi-u_h)^-\,dx.
\end{align*}
For any $T\in \bC_h\cup\,\bF_h$,
\begin{align*}
\int_T (\chi-u_h)^-\,dx &=\int_T \max\{u_h-\chi,0\}\,dx=\int_T
(u_h-\min\{u_h,\chi\})\,dx\\
&= \int_T (u_h-\chi_h)\,dx+\int_T (\chi_h-\min\{u_h,\chi\})\,dx.
\end{align*}
Hence
\begin{align}
\int_\Omega \bar \sigma_h(\chi-u_h)^-\,dx &=
\sum_{T\in\bC_h\cup\,\bF_h}\int_T \bar \sigma_h(u_h-\chi_h)\,dx
\label{eq:sigmalower}\\&\quad+
\sum_{T\in\bC_h\cup\,\bF_h}\int_T\bar
\sigma_h(\chi_h-\min\{u_h,\chi\})\,dx.\notag
\end{align}
Since
\begin{align*}
\sum_{T\in\bC_h}\int_T \bar \sigma_h (u_h-\chi_h)\,dx =0,
\end{align*}
we note that
\begin{align*}
\sum_{T\in\bC_h\cup\,\bF_h} \int_T \bar \sigma_h (u_h-\chi_h)\,dx
& = \sum_{T\in\bF_h} \int_T \bar \sigma_h (u_h-\chi_h)\,dx.
\end{align*}
Substitute this in \eqref{eq:sigmalower} and find
\begin{align*}
\int_\Omega \bar \sigma_h(\chi-u_h)^-\,dx &=
\sum_{T\in\bF_h}\int_T \bar \sigma_h(u_h-\chi_h)\,dx\\&\quad+
\sum_{T\in\bC_h\cup\,\bF_h}\int_T\bar
\sigma_h(\chi_h-\min\{u_h,\chi\})\,dx.
\end{align*}
This completes the proof.
\end{proof}

From Lemma \ref{lem:ErrorRelation}, \ref{lem:GhBound} and
\ref{lem:sigmaLower}, we deduce the following result on {\em a
posteriori} error control of quadratic fem:

\begin{theorem}\label{thm:Apost-Est} It holds that
\begin{align*}
\|\nabla(u-u_h)\|^2+\|\sigma-\sigma_h\|_{-1}^2&\leq
C\Big(\eta_1^2+\eta_2^2+\eta_3^2+\|\nabla(\chi-u_h)^+\|^2
-\sum_{T\in\bF_h}\int_T \bar \sigma_h(u_h-\chi_h)\,dx
\Big.\\&\qquad \Big.-\sum_{T\in\bC_h\cup\,\bF_h}\int_T\bar
\sigma_h(\chi_h-\min\{u_h,\chi\})\,dx\Big).
\end{align*}
\end{theorem}

\subsection{Simplified Error Estimator}
Motivated by the results in \cite{TG:2014:VIDG1}, we derive a
simplified error estimator  under an assumption on the trace of
the obstacle that $(\chi-\chi_h)|_{\partial\Omega}=0$, where
recall that $\chi_h\in V_h$ is the Lagrange nodal interpolation of
$\chi$. We assume this condition on the trace of $\chi$ for the
rest of this subsection. Define
\begin{align*}
\tilde K =\{v\in H^1_0(\Omega): v\geq \chi_h \text{ a.e. in }
\Omega\},
\end{align*}
and let $\tilde u \in \tilde K$ solves
\begin{align*}
a(\tilde u,\,v-\tilde u)\geq (f,\,v-\tilde u)\quad \forall \, v\in
\tilde K.
\end{align*}
Since $\chi_h^+\in\tilde K$, there exists a unique solution to the
above auxiliary problem. The result in \cite{TG:2014:VIDG1}
implies
\begin{align}\label{eq:aux}
\|\nabla(u-\tilde u)\|\leq C \|\nabla(\chi-\chi_h)\|.
\end{align}
Using the same arguments in proving Theorem \ref{thm:Apost-Est}
and replacing $\chi$ with $\chi_h$, we deduce the following
result:
\begin{lemma}\label{lem:Apost-Est1} There hold
\begin{align*}
\|\nabla(\tilde u-u_h)\|^2&\leq
C\Big(\eta_1^2+\eta_2^2+\eta_3^2+\|\nabla(\chi_h-u_h)^+\|^2
-\sum_{T\in\bF_h}\int_T \bar \sigma_h(\chi_h-u_h)^-\,dx
\Big.\\&\qquad \Big.-\sum_{T\in\bC_h}\int_T\bar
\sigma_h(\chi_h-u_h)^+\,dx\Big).
\end{align*}
\end{lemma}
Combining \eqref{eq:aux} and Lemma \ref{lem:Apost-Est1}, we obtain
the following result:
\begin{theorem}\label{thm:Apost-Est1} Let $(\chi-\chi_h)|_{\partial\Omega}=0$. Then
there hold
\begin{align*}
\|\nabla(u-u_h)\|^2&\leq
C\Big(\eta_1^2+\eta_2^2+\eta_3^2+\|\nabla(\chi_h-u_h)^+\|^2
+\|\nabla(\chi-\chi_h)\|^2 \Big.\\&\qquad
\Big.-\sum_{T\in\bF_h}\int_T \bar
\sigma_h(\chi_h-u_h)^-\,dx-\sum_{T\in\bC_h}\int_T\bar
\sigma_h(\chi_h-u_h)^+\,dx\Big).
\end{align*}
\end{theorem}

\begin{remark}
The difference between the estimator in Theorem
\ref{thm:Apost-Est1} and the estimator in Theorem
\ref{thm:Apost-Est} is that the min/max functions involve only
discrete functions. This provides simplicity in computations.
\end{remark}

\par
\noindent The efficiency of the error estimators $\eta_1$ and
$\eta_2$ will be followed in a similar way as in
\cite{Veeser:2001:VI}. Then the efficiency of the error estimator
$\eta_3$ is followed by the use of Lemma \ref{lem:sigmaosc}. The
efficiency of the other error estimators involving min/max
functions in Theorem \ref{thm:Apost-Est1}/\ref{thm:Apost-Est} is
less clear than in the case of linear finite element method in
\cite{Veeser:2001:VI}. This subject will be pursued in the future.

\section{A Priori Error Analysis}\label{sec:apriori}
In this section, we show under some regularity conditions that the
discrete function $\sigma_h$ converges to $\sigma$ with some rate
of convergence. Later on we derive optimal order error estimates
under an hypothesis on the contact set.

\par
\noindent First we derive some error estimates for $\sigma_h$.
\begin{theorem}\label{thm:sigma-L2-apriori}
Let $u,\,\chi \in H^{2+s}(\Omega)$ and $f\in H^s(\Omega)$ for
$0\leq s\leq 1$. Let $\sigma$ and $\sigma_h$ be defined by
$\eqref{eq:sigmadef}$ and $\eqref{eq:sigmahdef}$. Then, it holds
\begin{align*}
\|\sigma-\sigma_h\|_{L^2(\Omega)}\leq C \left( h^s \|f+\Delta
u\|_{H^s(\Omega)}+h^{-1}\|\nabla(u-u_h)\|\right).
\end{align*}
\end{theorem}
\begin{proof}
From the hypothesis $f+\Delta u \in H^s(\Omega)$, we have $\sigma
=f+\Delta u$ and $\sigma \in H^s(\Omega)$. Note that
\begin{align*}
\|\sigma-\sigma_h\|_{L^2(\Omega)} =\sup_{\phi\in
L^2(\Omega),\;\phi\neq 0}\frac{( \sigma-\sigma_h,\phi)}{\|\phi\|}.
\end{align*}
Let $\phi\in L^2(\Omega)$. Let $P_h\phi$ and $P_h\sigma$ be the
$L^2$-projections of $\phi$ and $\sigma$ on to $V_{nc}$,
respectively. A scaling argument implies that
$\|\Pi_hP_h\phi\|\leq C \|P_h\phi\|\leq C \|\phi\|$. Then we find
using \eqref{eq:sigmahdef} that
\begin{align*}
(\sigma-\sigma_h,\phi)&=(\sigma-\sigma_h,\phi-P_h\phi)+(\sigma-\sigma_h,P_h\phi)\\
&=(\sigma-P_h\sigma, \phi-P_h\phi)+(f+\Delta u,P_h\phi)-(\sigma_h,P_h\phi)\\
&=(\sigma-P_h\sigma, \phi-P_h\phi)+(f+\Delta u,P_h\phi)-(f,\Pi_hP_h\phi)+a(u_h,\Pi_hP_h\phi)\\
&=(\sigma-P_h\sigma, \phi-P_h\phi)+(f+\Delta u,P_h\phi-\Pi_hP_h\phi)+a(u_h-u,\Pi_hP_h\phi)\\
&=(\sigma-P_h\sigma, \phi-P_h\phi)+(f+\Delta u-\overline{(f+\Delta u)},\phi_h-\Pi_hP_h\phi)+a(u_h-u,\Pi_hP_h\phi)\\
 &\leq C h^s \|f+\Delta u\|_{H^s(\Omega)}\left(\|P_h\phi\|+\|\Pi_hP_h\phi\|\right)+C
\|\nabla(u-u_h)\|\;\|\nabla\Pi_hP_h\phi\|\\
&\leq C h^s \|f+\Delta
u\|_{H^s(\Omega)}\left(\|P_h\phi\|+\|\Pi_hP_h\phi\|\right)+C
h^{-1}\|\nabla(u-u_h)\|\;\|\Pi_hP_h\phi\|\\
&\leq C \left( h^s \|f+\Delta
u\|_{H^s(\Omega)}+h^{-1}\|\nabla(u-u_h)\|\right)\|\phi\|.
\end{align*}
This completes the proof.
\end{proof}
\par
Next we derive an {\em a priori} error estimate for the multiplier
$\sigma_h$ in $H^{-1}$-norm.

\begin{theorem}\label{thm:sigma-apriori}
Let $u,\,\chi \in H^{s+2}(\Omega)$ and $f\in H^s(\Omega)$ for
$0\leq s\leq 1$. Let $\sigma$ and $\sigma_h$ be defined by
$\eqref{eq:sigmadef}$ and $\eqref{eq:sigmahdef}$. Then, it holds
\begin{align*}
\|\sigma-\sigma_h\|_{H^{-1}(\Omega)}\leq C \left( h^{s+1}
\|f+\Delta u\|_{H^s(\Omega)}+\|\nabla(u-u_h)\|\right).
\end{align*}
\end{theorem}
\begin{proof}
Note that
\begin{align*}
\|\sigma-\sigma_h\|_{H^{-1}(\Omega)} =\sup_{\phi\in
H^1_0(\Omega),\;\phi\neq 0}\frac{(
\sigma-\sigma_h,\phi)}{\|\nabla\phi\|}.
\end{align*}
Let $\phi\in H^1_0(\Omega)$ and let $\phi_h$ be the
$L^2$-projection of $\phi$ on to $V_h$.  Then using
\eqref{eq:sigmahdef}, we find
\begin{align*}
(\sigma-\sigma_h,\phi)&=(\sigma-\sigma_h,\phi-\phi_h)+(\sigma-\sigma_h,\phi_h)\\
&=(\sigma-\sigma_h, \phi-\phi_h)+(f,\phi_h)-a(u,\phi_h)-(\sigma_h,\phi_h)\\
&=(\sigma-\sigma_h, \phi-\phi_h)+a(u_h-u,\phi_h)+(f,\phi_h)-a(u_h,\phi_h)-(\sigma_h,\phi_h)\\
&=(\sigma-\sigma_h, \phi-\phi_h)+a(u_h-u,\phi_h)+(\sigma_h,\Pi_h^{-1}\phi_h-\phi_h)\\
&=(\sigma-\sigma_h, \phi-\phi_h)+a(u_h-u,\phi_h)+(\sigma_h-\bar\sigma,\Pi_h^{-1}\phi_h-\phi_h)\\
&\leq C h \|\sigma-\sigma_h\|_{L^2(\Omega)}\|\nabla \phi\|+
\|\nabla(u-u_h)\|\;\|\nabla\phi_h\|+Ch\|\sigma_h-\bar\sigma\|\|\nabla\phi_h\|\\
&\leq C \left( h
\|\sigma-\sigma_h\|_{L^2(\Omega)}+\|\nabla(u-u_h)\|+h
\|\sigma-\bar\sigma\| \right)\|\nabla \phi\|\\
&\leq C \left( h
\|\sigma-\sigma_h\|_{L^2(\Omega)}+\|\nabla(u-u_h)\|+h^{s+1}
\|\sigma\|_{H^s(\Omega)} \right)\|\nabla \phi\|.
\end{align*}
Finally using Theorem \ref{thm:sigma-L2-apriori}, we complete the
proof.
\end{proof}

\par
The following {\em a priori} error estimate has been derived in
\cite{BHR:1977:VI} and \cite[Theorem 3.1]{Wang:2002:P2VI}:
\begin{align}\label{eq:uh-apriori}
\|\nabla(u-u_h)\|\leq C h^{\frac{3}{2}-\epsilon} \text{  for any
  } \epsilon>0,
\end{align}
assuming that the solution $u$ of \eqref{eq:MP} possesses the
regularity
\begin{align}\label{eq;uregu}
u\in W^{s,p}(\Omega)\quad 1<p<\infty,\quad  s<2+\frac{1}{p},
\end{align}
and the data satisfies $f\in L^\infty(\Omega)\cap H^1(\Omega)$ and
$\chi\in W^{3,3}(\bar\Omega)\cap W^{2,\infty}(\Omega)$.

\par
\noindent The following corollary is immediate from
\eqref{eq:uh-apriori}, Theorem \ref{thm:sigma-apriori} and
\ref{thm:sigma-L2-apriori}:

\begin{corollary}\label{cor:H1-Apriori1}
Let $u,\,\chi \in H^{2+s}(\Omega)$ and $f\in H^s(\Omega)$ for
$0\leq s\leq 1$. Further assume that $\eqref{eq;uregu}$ holds. Let
$\sigma$ and $\sigma_h$ be defined by $\eqref{eq:sigmadef}$ and
$\eqref{eq:sigmahdef}$. Then, there hold for any $\epsilon>0$
\begin{align*}
\|\sigma-\sigma_h\|_{L^2(\Omega)}&\leq C
\left(h^s\|f+\Delta u\|_{H^s(\Omega)}+C h^{1/2-\epsilon}\right),\\
\|\sigma-\sigma_h\|_{H^{-1}(\Omega)}&\leq C
\left(h^{1+s}\|f+\Delta u\|_{H^s(\Omega)}+C
h^{3/2-\epsilon}\right).
\end{align*}
\end{corollary}

\par
\noindent The error estimates in Corollary \ref{cor:H1-Apriori1}
are suboptimal due to the suboptimal estimate in
\eqref{eq:uh-apriori}.

\subsection{A Priori Error Analysis: Revisited}
Recently in \cite[Theorem 4.2]{WHC:2010:DGVI}, an {\em a priori}
error estimate of order $h^{3/2}$ for a quadratic DG method for
the obstacle problem has been derived assuming that the obstacle
$\chi \in H^3(\Omega)$, the force $f\in H^1(\Omega)$ and the
solution $u\in H^3(\Omega)$. We revisit the analysis under this
regularity if an optimal order error estimates may be derived.
\par
\noindent
\smallskip
The following lemma is well-known
\cite{Glowinski:2008:VI,KS:2000:VI}:

\begin{lemma}
Let $u,\,\chi\in H^2(\Omega)$ and $f\in L^2(\Omega)$. Then there
hold
\begin{enumerate}
\item $\Delta u+f \leq 0$ a.e. in  $\Omega$,

\item $-\Delta u=f$ a.e. on the set $\{u>\chi\}$,

\item $(\Delta u+f,u-\chi)=0$.
\end{enumerate}
\end{lemma}

\par
\noindent The error analysis below is based on the following sets:
\begin{align*}
\Omega_h^c&=\cup\{T\in \cT_h: u=\chi\text{ on } T\},\\
\Omega_h^n&=\cup\{T\in \cT_h: u >\chi\text{ on } \text{int}(T)\},\\
\Omega_h^f&=\Omega\backslash(\Omega_h^c\cup\Omega_h^n),
\end{align*}
where $\text{int}(T)$ is the interior of  $T$.

\par
For the rest of the article, we assume the following:
\par
\noindent {\bf Assumption (F)}: We assume that in each $T \subset
\bar\Omega_h^f$, there is a neighborhood $O\subset T$ such that
$u\equiv\chi$ in $O$.

\par
The above assumption $({\bf F})$ means that the contact set
$\{u\equiv \chi\}$ does not degenerate to a curve in any part of
the domain $\Omega$.

\begin{lemma}\label{lem:PoincareFree}
Let the assumption $({\bf F})$ holds and let $u,\,\chi \in
H^3(\Omega)$. Then for any $T\in \cT_h$ with $T\subset
\bar\Omega_h^f$, it holds that
$$\|u-\chi\|_{H^1(T)}\leq C h_T^2 |u-\chi|_{H^3(T)}.$$
\end{lemma}
\begin{proof}
The imbedding $H^3(\Omega)\subset C^1(\bar\Omega)$ implies that
$u-\chi\in C^1(\bar\Omega)$. For any $T\subset \bar\Omega_h^f$,
there exists a neighborhood $O\subset T$ such that $(u-\chi)\equiv
0$ on $O$. This implies all the weak derivatives
$D^\alpha(u-\chi)\equiv 0$ on $O$ (a.e.) for $|\alpha|\leq 3$.
Using compactness and scaling arguments, it can be easily shown
that
\begin{align*}
|u-\chi|_{H^1(T)}\leq C h_T^2 |u-\chi|_{H^3(T)}.
\end{align*}
This completes the proof.
\end{proof}
\par
Below, we denote by $I_hu \in V_h$, the standard $P_2$-Lagrange
interpolation of $u$. We now prove the optimal error estimate.
\begin{theorem}\label{thm:H1-Apriori}
Let the assumption $({\bf F})$ holds  and let $u,\,\chi \in
H^3(\Omega)$ and $f\in H^1(\Omega)$. Then, it holds
\begin{align*}
\|\nabla(u-u_h)\| \leq C h^2
\left(\|\chi\|_{H^3(T)}+\|u\|_{H^3(T)}+\|f\|_{H^1(\Omega)}\right).
\end{align*}
\end{theorem}
\begin{proof}
Since $I_hu\in K_h$, we note using \eqref{eq:DP} that
\begin{align*}
\|\nabla(u-u_h)\|^2&=
a(u-u_h,u-u_h)=a(u-u_h,u-I_hu)+a(u-u_h,I_hu-u_h)\\
&\leq a(u-u_h,u-I_hu)+a(u,I_hu-u_h)-(f,I_hu-u_h)\\
&=a(u-u_h,u-I_hu)-(\Delta u+f,I_hu-u_h).
\end{align*}
Let $\sigma=\Delta u+f$. Then $\sigma\in L^2(\Omega)$. Since
$\sigma=0$ on $\Omega_h^n$, we have
\begin{align*}
(\Delta u+f,I_hu-u_h)&=\sum_{T\subset\Omega_h^c} \int_T \sigma
(I_hu-u_h)\,dx+\sum_{T\subset\Omega_h^f}\int_T \sigma
(I_hu-u_h)\,dx.
\end{align*}
For $T\subset\Omega_h^c$, we have $u\equiv \chi$ on $T$ and hence
\begin{align*}
\int_T \sigma (I_hu-u_h)\,dx &=\int_T \sigma (I_h\chi-u_h)\,dx
\leq \int_T
(\sigma-\bar \sigma)(I_h\chi-u_h)\,dx\\
&=\int_T (\sigma-\bar \sigma)(I_hu-u_h)\,dx\\
&=\int_T (\sigma-\bar \sigma)\big((I_hu-u_h)-\overline{(I_hu-u_h)}\big)\,dx\\
&\leq C h_T^2 \|\nabla \sigma\|_{L^2(T)} \|\nabla
(I_hu-u_h)\|_{L^2(T)}
\end{align*}
For $T\subset\Omega_h^f$, we have
\begin{align*}
\int_T \sigma (I_hu-u_h)\,dx &=\int_T \sigma
(I_hu-u+u-\chi+\chi-I_h\chi+I_h\chi-u_h)\,dx
\end{align*}
Since $\sigma\equiv 0$ on a subset of $T$ of measure nonzero, we
have as in Lemma \ref{lem:PoincareFree} that
\begin{align*}
\int_T \sigma (I_hu-u+\chi-I_h\chi)\,dx &\leq C h_T^4 \|\nabla
\sigma\|_{L^2(T)} \left(\|u\|_{H^3(T)}+\|\chi\|_{H^3(T)}\right).
\end{align*}
Note that for $T\subset\Omega_h^f$ (indeed for any $T$), we have
\begin{align*}
\int_T \sigma (u-\chi)\,dx=0.
\end{align*}
Finally,
\begin{align*}
\int_T \sigma (I_h\chi-u_h)\,dx &\leq \int_T
(\sigma-\bar \sigma)(I_h\chi-u_h)\,dx\\
&=\int_T (\sigma-\bar \sigma)\big((I_h\chi-u_h)-\overline{(I_h\chi-u_h)}\big)\,dx\\
&\leq C h_T^2 \|\nabla \sigma\|_{L^2(T)} \|\nabla
(I_h\chi-u_h)\|_{L^2(T)},
\end{align*}
and then by using triangle inequality, Lemma
\ref{lem:PoincareFree} and the interpolation estimates for $I_h$,
we find
\begin{align*}
\|\nabla (I_h\chi-u_h)\|_{L^2(T)} &\leq  \|\nabla
(I_h\chi-\chi)\|_{L^2(T)}+\|\nabla(\chi-u)\|_{L^2(T)}+\|\nabla
(u-u_h)\|_{L^2(T)}\\
&\leq C h_T^2
\left(\|\chi\|_{H^3(T)}+\|u-\chi\|_{H^3(T)}\right)+\|\nabla
(u-u_h)\|_{L^2(T)}.
\end{align*}
Therefore
$$
\|\nabla(u-u_h)\| \leq C h^2
\left(\|\chi\|_{H^3(T)}+\|u\|_{H^3(T)}+\|f\|_{H^1(\Omega)}\right).
$$
This completes the proof.
\end{proof}

We deduce the following corollary  using the results in Theorems
\ref{thm:H1-Apriori}, \ref{thm:sigma-apriori} and
\ref{thm:sigma-L2-apriori}.
\begin{corollary}\label{cor:H1-Apriori}
Let the assumption $({\bf F})$ holds and let $u,\,\chi \in
H^3(\Omega)$ and $f\in H^1(\Omega)$. Let $\sigma$ and $\sigma_h$
be defined by $\eqref{eq:sigmadef}$ and $\eqref{eq:sigmahdef}$.
Then, there hold
\begin{align*}
\|\sigma-\sigma_h\|_{L^2(\Omega)}&\leq C h
\left(\|\chi\|_{H^3(T)}+\|u\|_{H^3(T)}+\|f\|_{H^1(\Omega)}\right),\\
\|\sigma-\sigma_h\|_{H^{-1}(\Omega)}&\leq C h^2
\left(\|\chi\|_{H^3(T)}+\|u\|_{H^3(T)}+\|f\|_{H^1(\Omega)}\right).
\end{align*}
\end{corollary}

\section{Numerical Experiments}\label{sec:Numerics}
In this section, we discuss some numerical experiments using two
model problems.
\par
\noindent {\bf Model Example 1:} Let $\Omega=(-1.5,1.5)^2$,
$f=-2$, $\chi:=0$ and $u=r^2/2-\text{ln}(r)-1/2$ on $\p\O$, where
$r^2=x^2+y^2$ for $(x,y)\in \R^2$. Then the exact solution $u$ is
given by
\begin{equation*}
u := \left\{ \begin{array}{ll} r^2/2-\text{ln}(r)-1/2, &  \text{ if }  r\geq 1\\\\
0, &  \text{ if }  r<1.
\end{array}\right.
\end{equation*}

\par
\noindent {\bf Model Example 2:} Let $\Omega$ be the square with
corners $\{(-1,0),(0,-1),(1,0),(0,1)\}$ and the obstacle function
to be $\chi=1-2r^2$, where $r=\sqrt{x^2+y^2}$. The load function
$f$ is taken to be
\begin{equation*}
f(r) := \left\{ \begin{array}{ll} 0 & \text{ if } r< r_0,\\\\
4r_0/r & \text{ if } r\geq r_0,
\end{array}\right.
\end{equation*}
so that the solution $u$ takes the form
\begin{equation*}
u(r):= \left\{ \begin{array}{ll} 1-2r^2 & \text{ if } r < r_0,\\\\
4r_0(1-r) & \text{ if } r \geq r_0,
\end{array}\right.
\end{equation*}
where $r_0=(\sqrt{2}-1)/\sqrt{2}$.

\par
\noindent The model problem \eqref{eq:MP} is considered in the
analysis with homogeneous boundary condition for avoiding
additional technical difficulties. However the error analysis in
the paper is still valid up to some higher order terms involving
the nonhomogeneous boundary condition.

\par
\noindent Firstly, we test the order of convergence under the
uniform refinement. Since the exact solutions are not
$H^3(\Omega)$ regular, the energy norm error will be convergent at
suboptimal rate. This can be clearly seen in the Tables
\ref{table:Ex1} and \ref{table:Ex2}. However we will see in the
numerical experiments using adaptive refinement that the errors
converge with optimal order ($1/N$, where $N$=number of degrees of
freedom). This demonstrates the optimal performance of the
quadratic fem for obstacle problem.

\begin{table}[h!!]
 {\small{\footnotesize
\begin{center}
\begin{tabular}{|c|c|c|c|c|c|c|c|}\hline
   $h$ & $\|\nabla(u-u_h)\| $  & order of conv.\\
\hline\\[-12pt]  &&\\[-8pt]
3/4  &0.359703822003801 & --     \\
3/8  &0.127058164618133 &1.501       \\
3/16 &0.058540022081520 &1.117       \\
3/32 &0.017334877653178 &1.755   \\
3/64 &0.004870365957461 &1.831     \\
3/128&0.001950822843142 &1.319      \\
3/256&0.000781008462447 &1.320   \\
\hline
\end{tabular}
\end{center}
}}
\par\medskip
\caption{Error and orders of convergence for Example 1 }
\label{table:Ex1}
\end{table}

\begin{table}[h!!]
 {\small{\footnotesize
\begin{center}
\begin{tabular}{|c|c|c|c|c|c|c|c|}\hline
   $h$ & $\|\nabla(u-u_h)\| $  & order of conv.\\
\hline\\[-12pt]  &&\\[-8pt]
3/4  &0.206211469561149 & --     \\
3/8  &0.058342275497902 &1.821       \\
3/16 &0.025124856349493 &1.215       \\
3/32 &0.007971135017375 &1.656   \\
3/64 &0.002583671405642 &1.625     \\
3/128&0.000931014496396 &1.472      \\
3/256&0.000323678406046 &1.524   \\
\hline
\end{tabular}
\end{center}
}}
\par\medskip
\caption{Error and orders of convergence for Example 2}
\label{table:Ex2}
\end{table}

We now conduct tests on adaptive algorithm. For this, we consider
an initial mesh with four right-angled cris-cross mesh for both
the examples. Then we use the adaptive algorithm consisting of
four successive modules
\begin{equation*}
{\bf SOLVE}\rightarrow  {\bf ESTIMATE} \rightarrow {\bf
MARK}\rightarrow {\bf REFINE}
\end{equation*}
We use the primal-dual active set strategy
\cite{HK:2003:activeset} in the step SOLVE to solve the discrete
obstacle problem. The estimator in Theorem \ref{thm:Apost-Est1} is
computed in the step ESTIMATE and then the D\"orfler's marking
strategy \cite{Dorfler:1996:Afem} with parameter $\theta=0.3$ has
been used in the step MARK to mark the elements for refinement.
Using the newest vertex bisection algorithm, we refine the mesh
and obtain a new mesh.

\par
\noindent The convergence history of errors and estimators is
depicted in Figure \ref{fig:ErrEst1} and \ref{fig:ErrEst2} for
Example 1 and 2, respectively. These figures illustrate the
optimal order convergence as well as the reliability of the error
estimator. The efficiency indices can be seen in Figures
\ref{fig:EI1} and \ref{fig:EI2}. The free boundary sets for both
the examples have been captured by the error estimator very
efficiently, see Figures \ref{fig:Mesh1} and \ref{fig:Mesh2}.

\par
\noindent {\bf Heuristic comments on the optimal order
convergence.} In our first experiment using uniform refinement, we
found only suboptimal rate of convergence due to lack of the
regularity of the solutions. It is well known that the adaptive
schemes restore the optimal rate of the method even for the
problems with irregular solutions. We find the same in our
experiments. Heuristically this explains the optimal rate {\em a
priori} error estimates in the Section \ref{sec:apriori}.

\begin{figure}[t]
\begin{center}
\includegraphics*[width=10cm,height=7cm]{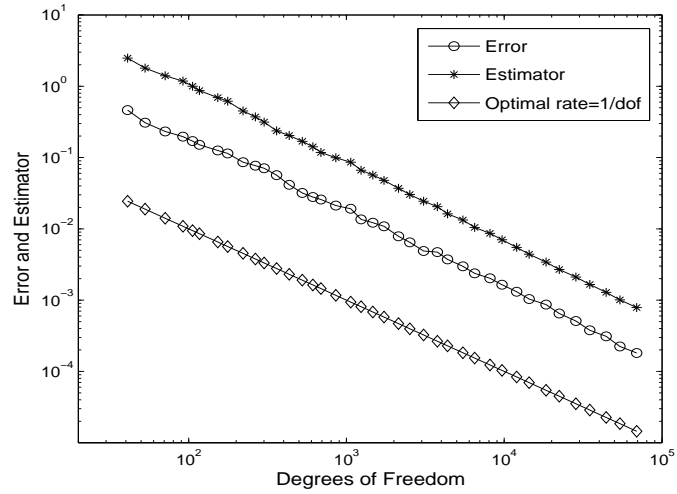}
\caption{{\large Errors and Estimators for Example 1}}
\label{fig:ErrEst1}
\end{center}
\end{figure}

\begin{figure}[t]
\begin{center}
\includegraphics*[width=10cm,height=7cm]{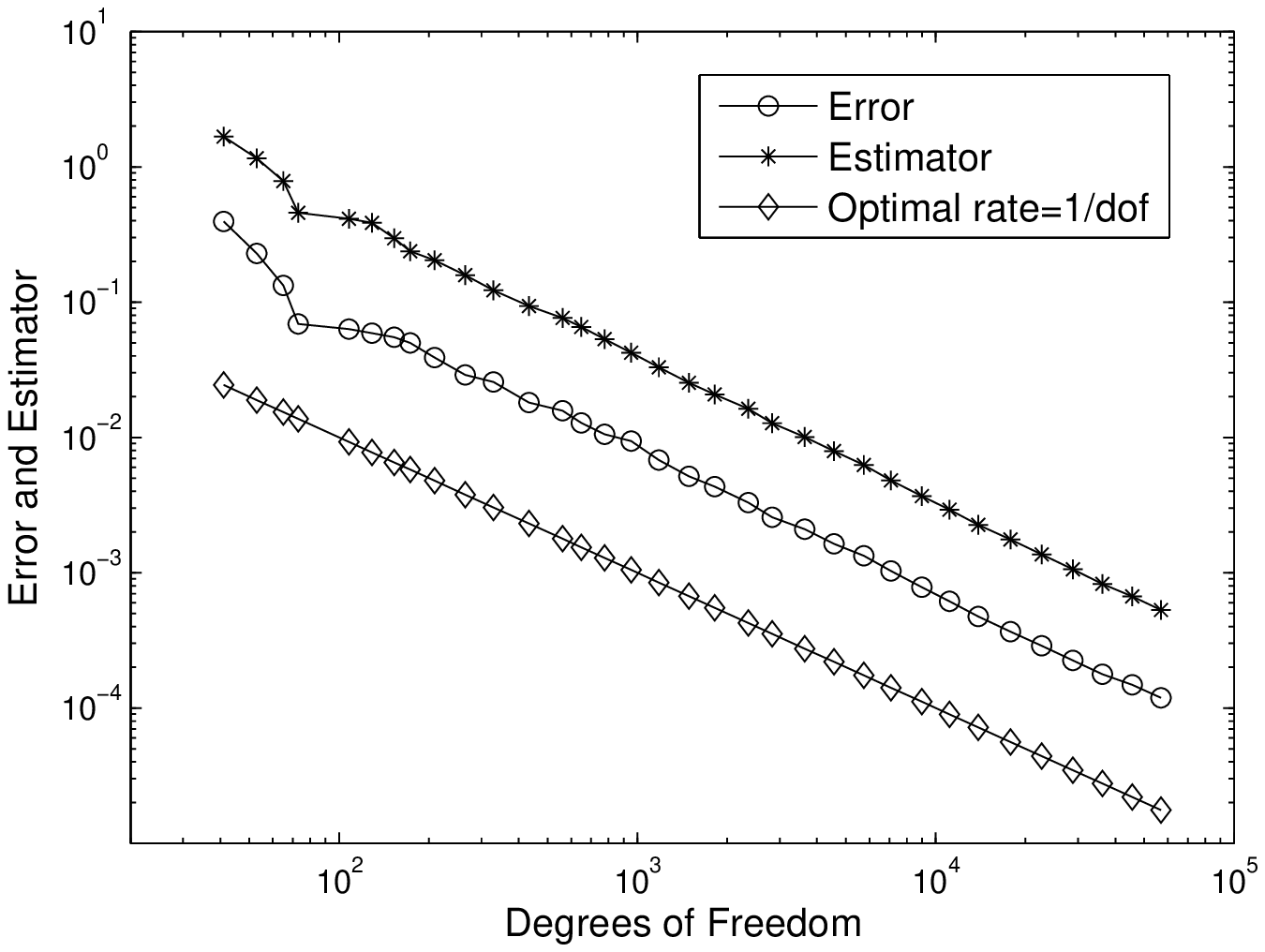}
\caption{{\large Errors and Estimators for Example 2}}
\label{fig:ErrEst2}
\end{center}
\end{figure}

\begin{figure}[t]
\begin{center}
\includegraphics*[width=10cm,height=7cm]{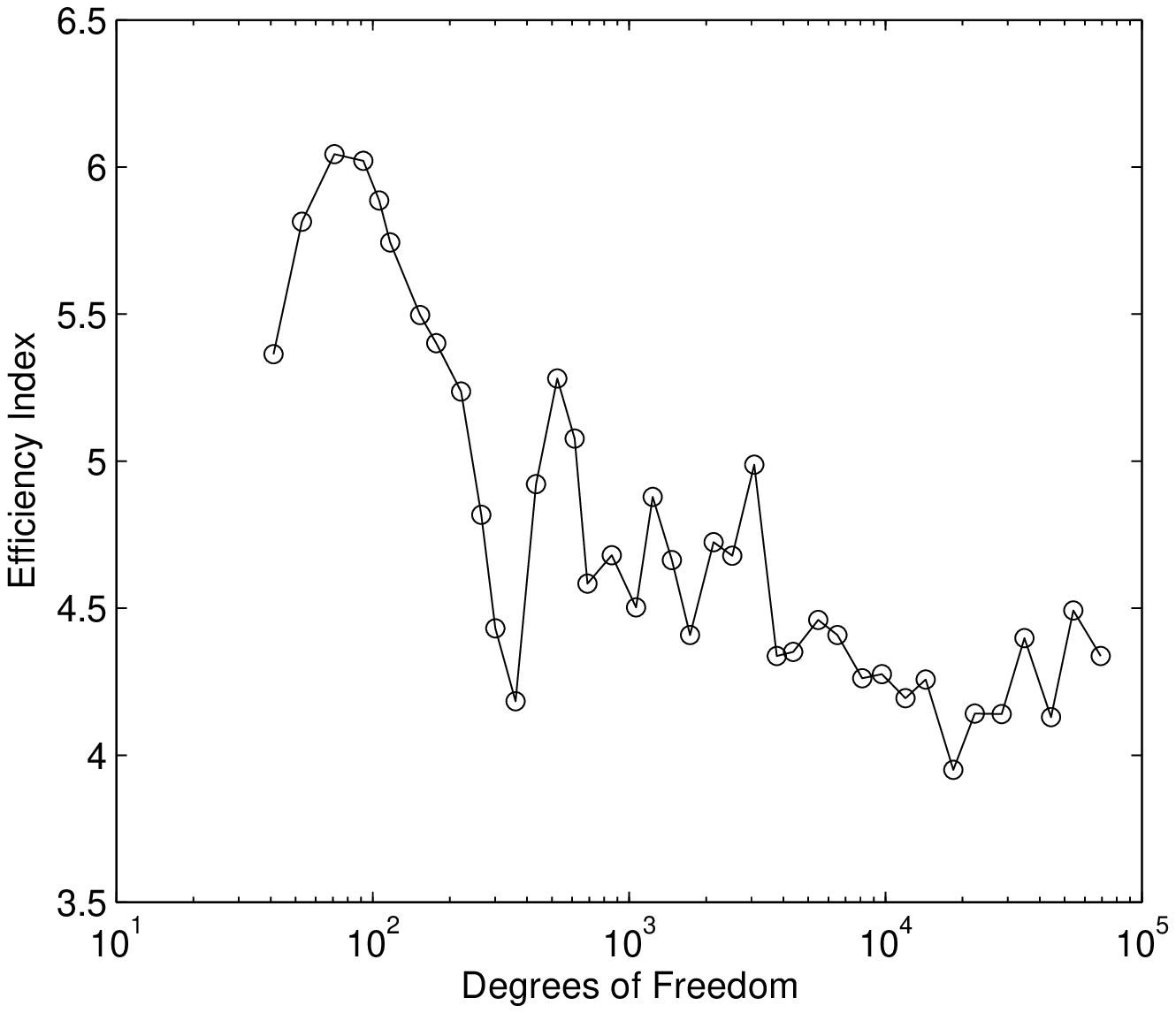}
 \caption{{\large Efficiency Index for Example 1}} \label{fig:EI1}
\end{center}
\end{figure}

\begin{figure}[t]
\begin{center}
\includegraphics*[width=10cm,height=7cm]{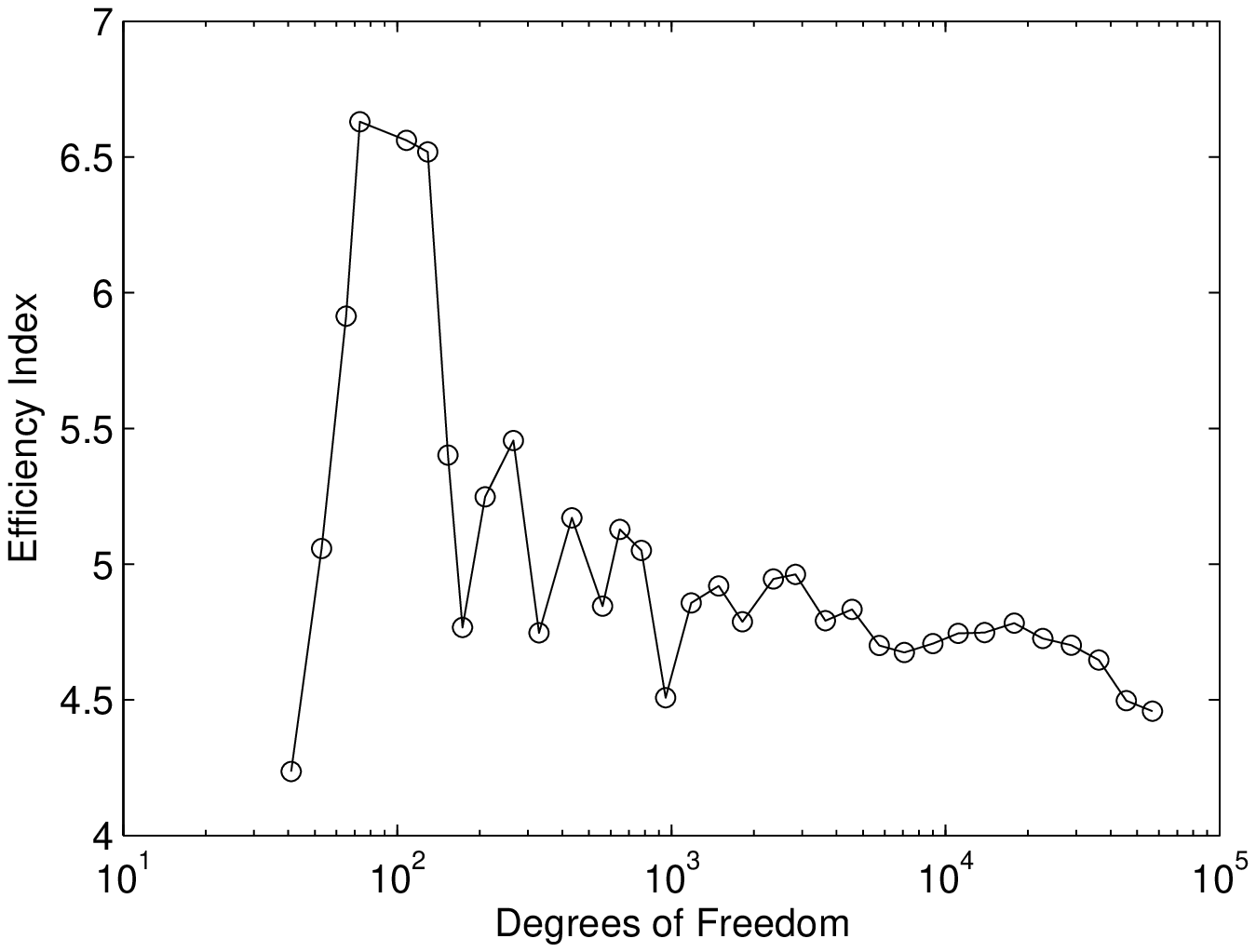}
 \caption{{\large Efficiency Index for Example 2}} \label{fig:EI2}
\end{center}
\end{figure}

\begin{figure}[t]
\begin{center}
\includegraphics*[width=10cm,height=7cm]{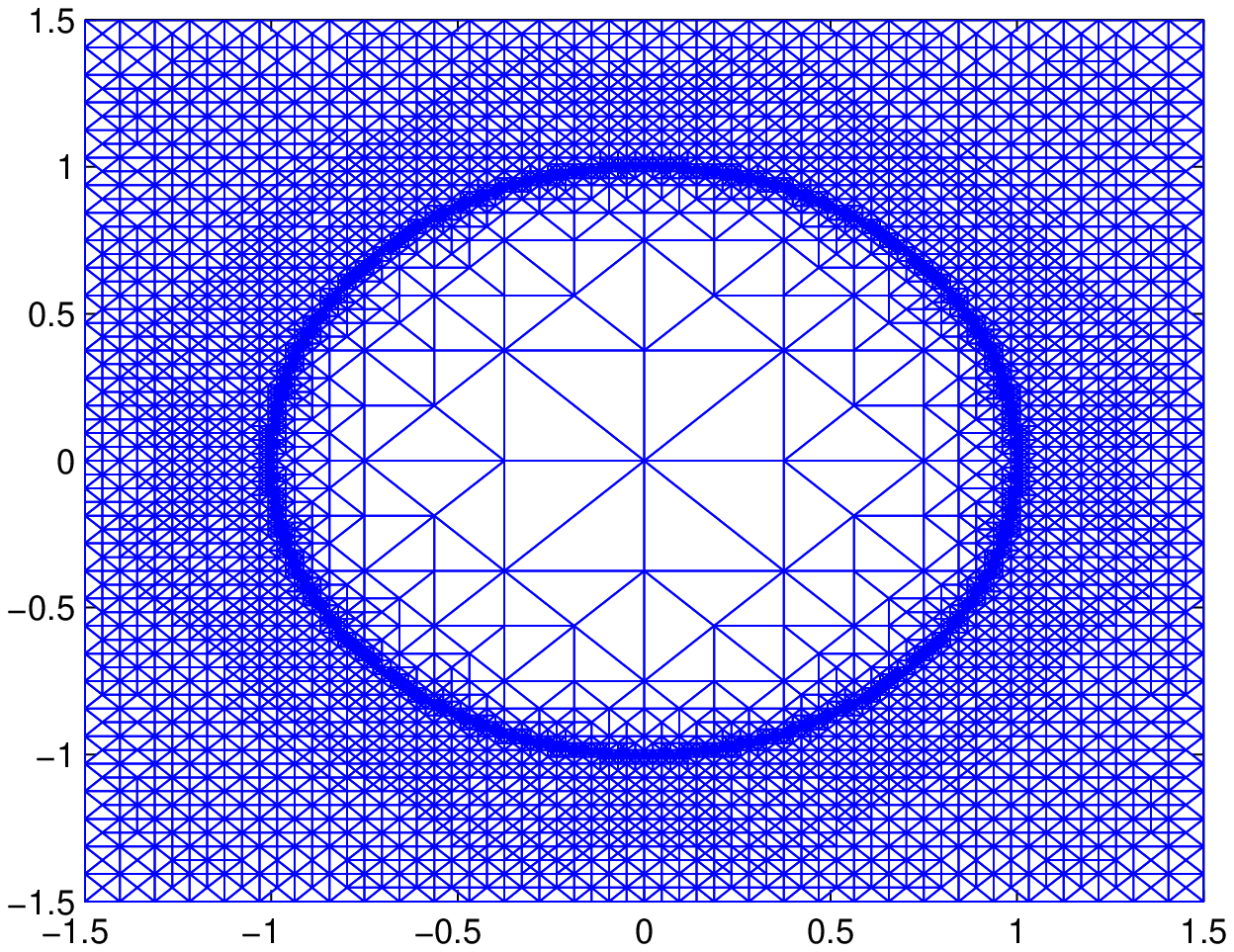}
 \caption{{\large Mesh at intermediate level for Example 1}} \label{fig:Mesh1}
\end{center}
\end{figure}

\begin{figure}[t]
\begin{center}
\includegraphics*[width=10cm,height=7cm]{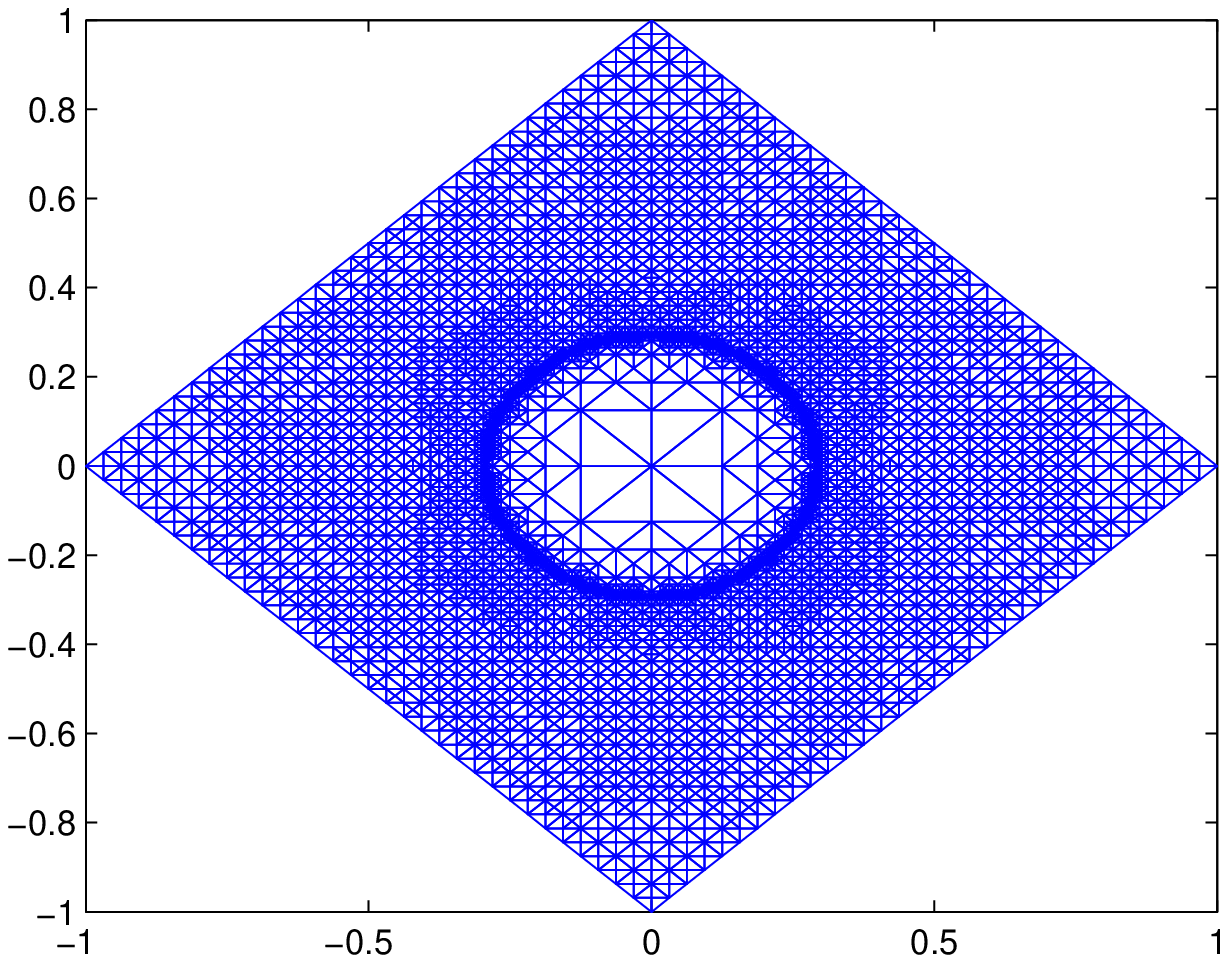}
 \caption{{\large Mesh at intermediate level for Example 2}} \label{fig:Mesh2}
\end{center}
\end{figure}

\section{Conclusions}\label{sec:Conclusions}
For the first time, residual based {\em a posteriori} error
estimator has been derived for the quadratic finite element method
for the elliptic obstacle problem. The estimator is shown to be
reliable. The efficiency of the error estimator in this case is
less clear than in the case of linear fem, we leave this subject
to future investigation. The error estimator involves a discrete
Lagrange multiplier which is shown to be optimally convergent to
the continuous one whenever the solution $u$, obstacle $\chi$ and
the force $f$ are sufficiently smooth and the contact set does not
degenerate to a curve in any part of the domain. Also under this
assumption, we show that the quadratic fem for obstacle problem is
indeed optimal. Numerical experiments with adaptive refinement
exhibit this optimal convergence rate.

%


\begin{thebibliography}{10}
%
\bibitem{AO:2000:Book}
M.~Ainsworth and J.~T. Oden.
\newblock {\em A posteriori error estimation in finite element analysis}.
\newblock Pure and Applied Mathematics (New York). Wiley-Interscience [John
  Wiley \& Sons], New York, 2000.


\bibitem{AH:2009:VI}
K. Atkinson and W. Han.
\newblock {\em Theoretical Numerical Analysis. A functional analysis framework}.
Thrid edition, Springer, 2009.


\bibitem{BC:2004:VI}
S. Bartels and C. Carstensen.
\newblock Averaging techniques yield relaible a posteriori finite element error control for obstacle problems,
{\em Numer. Math.}, 99:225--249, 2004.


\bibitem{Braess:2005:VI}
D. Braess.
\newblock {A posteriori error estimators for obstacle problems-another look}.
\newblock {\em Numer. Math.}, 101:415-421, 2005.


\bibitem{BScott:2008:FEM}
S.C. Brenner and L.R. Scott.
\newblock {\em {The Mathematical Theory of Finite Element Methods $($Third
  Edition$)$}}.
\newblock Springer-Verlag, New York, 2008.

\bibitem{BHR:1977:VI}
F. Brezzi, W. W. Hager, and P. A. Raviart.
\newblock {Error estimates for the finite element solution of variational inequalities}, Part I. Primal theory.
\newblock {\em Numer. Math.}, 28:431--443, 1977.

\bibitem{BCH:2007:VI}
D. Braess, C. Carstensen and R.H.W. Hoppe.
\newblock {Convergence analysis of a conforming adaptive finite element method for an obstacle problem}.
\newblock {\em Numer. Math.}, 107:455--471, 2007.


\bibitem{CN:2000:VI}
Z. Chen and R. Nochetto.
\newblock {Residual type a posteriori error estimates for elliptic obstacle problems}.
\newblock {\em Numer. Math.}, 84:527--548, 2000.


\bibitem{Ciarlet:1978:FEM}
P.G. Ciarlet.
\newblock {\em {The Finite Element Method for Elliptic Problems}}.
\newblock North-Holland, Amsterdam, 1978.


\bibitem{Dorfler:1996:Afem}
W. D\"orlfer
\newblock A convergent adaptive algorithm for Poisson's equation.
\newblock {\em SIAM J. Numer. Anal.}, 33:1106--1124, 1996.

\bibitem{Falk:1974:VI}
R. S. Falk. \newblock Error estimates for the approximation of a
class of variational inequalities, {\em  Math. Comp}., 28:963-971,
(1974).


\bibitem{Glowinski:2008:VI}
R. Glowinski.
\newblock {\em Numerical Methods for Nonlinear Variational
Problems}. Springer-Verlag, Berlin, 2008.

\bibitem{TG:2014:VIDG}
T. Gudi and K. Porwal. \newblock A posteriori error control of
discontinuous Galerkin methods for elliptic obstacle problems,
{\em Math. Comput.}, 83:579--602, 2014.

\bibitem{TG:2014:VIDG1}
T. Gudi and K. Porwal. \newblock A remark on the a posteriori
error analysis of discontinuous Galerkin methods for obstacle
problem, {\em Comput. Meth. Appl. Math.}, 14:71--87, 2014.


\bibitem{HK:2003:activeset}
M. Hinterm\"uller, K. Ito and K. Kunish.
\newblock{The primal-dual active set strategy as a semismooth Newton
method.}
\newblock {\em SIAM J. Optim.}, 13:865--888, 2003.


\bibitem{KS:2000:VI}
D. Kinderlehrer and G. Stampacchia. \newblock {\em  An
Introduction to Variational Inequalities and Their Applications}.
\newblock SIAM, Philadelphia, 2000.

\bibitem{NPZ:2010:VI}
R. Nochetto, T. V. Petersdorff and C. S. Zhang.
\newblock {A posteriori error analysis for a class of integral equations and variational inequalities}.
\newblock {\em Numer. Math}, 116:519--552, 2010.


\bibitem{Veeser:2001:VI}
A. Veeser.
\newblock {Efficient and Relaible a posteriori error estimators for elliptic obstacle problems}.
\newblock {\em SIAM J. Numer. Anal.}, 39:146--167, 2001.


\bibitem{Wang:2002:P2VI}
L. Wang.
\newblock {On the quadratic finite element approximation to the obstacle problem.}
\newblock {\em Numer. Math.}, 92:771--778, 2002.

\bibitem{WHC:2010:DGVI}
F. Wang, W. Han and X.Cheng.
\newblock {Discontinuous Galerkin methods for solving elliptic variational inequalities}.
\newblock {\em SIAM J. Numer. Anal.}, 48:708--733, 2010.


\bibitem{WW:2010:Apost}
A. Weiss and B. I. Wohlmuth
\newblock {A posteriori error estimator for obstacle problems}.
\newblock {\em SIAM J. Numer. Anal.}, 32:2627--2658, 2010.



\end{thebibliography}
\end{document}